\documentclass{amsart}
\usepackage{amssymb}
\usepackage{amsmath}
\usepackage{tikz-cd}
\usepackage{amsthm}
\theoremstyle{plain}

\theoremstyle{plain}
\newtheorem{thm}{Theorem}
\newtheorem{prop}{Proposition}
\newtheorem{lem}{Lemma}
\newtheorem{coro}{Corollary}

\usepackage{geometry}
\theoremstyle{definition}

\newtheorem{rem}{Remark}

\begin{document}
\setcounter{page}{1}

\title[$AM$-spaces from a locally solid vector lattice point of view with applications]{ $AM$-spaces from a locally solid vector lattice point of view with applications}

\author[Omid Zabeti ]{Omid Zabeti}

\address{ Department of Mathematics, University of Sistan and Baluchestan, P.O. Box: 98135-674, Zahedan, Iran.}
\email{{o.zabeti@gmail.com}}



\date{Received: xxxxxx; Revised: yyyyyy; Accepted: zzzzzz.}

\begin{abstract}
Suppose $X$ is a locally solid vector lattice. In this paper, we introduce the notion "$AM$-property" in $X$ as an extension for $AM$-spaces in the category of all Banach lattices. With the aid of this concept, we characterize spaces in which bounded sets and order bounded sets agree. This, in turn, characterizes conditions under which each class of bounded operators on $X$ is order bounded and vice versa. Also, we show that under some natural assumptions, different types of bounded order bounded operators on $X$ have the Lebesgue or Levi property if and only if so is $X$.
\end{abstract}
\date{\today}
\maketitle

\section{motivation and Preliminaries}
Let us start with some motivation. Suppose $K$ is a compact Hausdorff topological space and $C(K)$ is the Banach lattice of all real-valued continuous functions on $K$. In general, $C(K)$-spaces have important properties in the category of all Banach lattices; for example in a $C(K)$-space, the Levi property and order completeness agree, order boundedness and boundedness coincide and so on. In addition, every $AM$-space can be isometrically embedded into a $C(K)$-space ( the remarkable Kakutani theorem); recall that a Banach lattice $E$ which satisfies $\|x\vee y\|=\|x\|\vee\|y\|$ for every $x,y\in E_{+}$ is termed an $AM$-space. On the other hand, one interesting question regarding operators between normed lattices is the following.

Is there any relation between continuous operators and order bounded ones?

It is known that when the domain is a Banach lattice, every order bounded operator is continuous. In general, the answer for the other direction is almost negative. Nevertheless, in some special cases, this can happen; when $E$ and $F$ are $C(K)$-spaces, then an operator $T:E\to F$ is continuous if and only if it is order bounded;  note that we can not replace $C(K)$-space with an $AM$-space ( see \cite[Example 4.73]{AB}). Furthermore, in a locally solid vector lattice, there are several non-equivalent ways to define bounded operators. So, it is interesting to investigate the relations between these types of bounded operators and order bounded ones. This is what our paper is about. Moreover, it is shown that in \cite{EGZ}, under some mild assumptions, each class of bounded order bounded operators between locally solid vector lattices, can have lattice structure, too. So, another interesting direction is to find situations under which a known property regarding lattice structure such as the Levi or Lebesgue property can be transformed between the space of bounded operators and the underlying space. This is another application of $AM$-property. More explicitly, we show that each class of bounded operators between locally solid vector lattices, under some conditions related to $AM$-property, possesses the Levi or Lebesgue property if and only if so is the range space.

Now, we consider some preliminaries which are needed in the sequel. For undefined terminology and also related notions as well as for a background on locally solid vector lattices, the reader is referred to \cite{AB1,AB}.

A vector lattice $X$ is called {\bf order complete} if every non-empty bounded above subset of $X$ has a supremum. $X$ is {\bf Archimedean} if $nx\leq y$ for each $n\in \Bbb N$ implies that $x\leq 0$. It is known that every order complete vector lattice is Archimedean. A set $S\subseteq X$ is called {\bf solid} if $x\in X$, $y\in S$ and $|x|\leq |y|$ imply that $x\in S$. Also, recall that a linear topology $\tau$ on a vector lattice $X$ is referred to as  {\bf locally solid} if it has a local basis at zero consisting of solid sets.

Suppose $X$ is a locally solid vector lattice. A net $(x_{\alpha})\subseteq X$ is said to be {\bf order} convergent to $x\in X$ if there exists a net $(z_{\beta})$ ( possibly over a different index set) such that $z_{\beta}\downarrow 0$ and for every $\beta$, there is an $\alpha_0$ with $|x_{\alpha}-x|\leq z_{\beta}$ for each $\alpha\ge \alpha_0$. A set $A\subseteq X$ is called {\bf order closed} if it contains limits of all order convergent nets which lie in $A$.

 Keep in mind that topology $\tau$ on a locally solid vector lattice $(X,\tau)$ is referred to as {\bf Fatou} if it has a local basis at zero consisting of solid order closed neighborhoods. In this case, we say that $X$ has the Fatou property.

 Observe that a locally solid vector lattice $(X,\tau)$ is said to have the {\bf Levi property} if every $\tau$-bounded upward directed set in $X_{+}$ has a supremum.

 Finally, recall that a locally solid vector lattice $(X,\tau)$ possesses the {\bf Lebesgue property} if for every net $(u_{\alpha})$ in $X$, $u_{\alpha}\downarrow 0$ implies that $u_{\alpha}\xrightarrow{\tau}0$.

Let us fix a convention. Suppose $(X,\tau)$ is a locally solid vector lattice. So, it has a local basis at zero consisting of solid sets. In this paper, we always choose zero neighborhoods from this basis. Moreover, if $X$ possesses the Fatou property, it contains a local basis consisting of solid order closed sets. So, we always pick zero neighborhoods from this basis.

For undefined expressions and related topics, see \cite{AB1}.

\section{main results}

{\bf Observation}. Suppose $X$ is an Archimedean vector lattice. For every subset $A$, by $A^{\vee}$, we mean the set of all finite suprema of elements of $A$; more precisely,
$A^{\vee}=\{a_1\vee\ldots\vee a_n: n\in \Bbb N, a_i\in A\}$. It is obvious that $A$ is bounded above in $X$ if and only if so is $A^{\vee}$ and in this case, when the supremum exists, $\sup A=\sup A^{\vee}$. Moreover, put $A^{\wedge}=\{a_1\wedge\ldots\wedge a_n: n\in \Bbb N, a_i\in A\}$. It is easy to see that $A$ is bounded below if and only if so is $A^{\wedge}$ and $\inf A=\inf A^{\wedge}$ ( when the infimum exists). Observe that  $A^{\vee}$ can be viewed as an upward directed set in $X$ and $A^{\wedge}$ can be considered as a downward directed set.

Suppose $X$ is a locally solid vector lattice. We say that $X$ has {\bf $AM$-property} provided that for every bounded set $B\subseteq X$, $B^{\vee}$ is also bounded. It is worthwhile to mention that when $B$ is bounded and solid, $B^{\vee}$ is bounded if and only if $B^{\wedge}$ is bounded; this follows from the fact that $X$ is locally solid and $x_1\wedge\ldots\wedge x_n=-((-x_1)\vee\ldots\vee(-x_n))$ for any $n\in \Bbb N$ and any $x_i\in B$.

In prior to anything, we show that $AM$-property is the "right" extension for $AM$-spaces. Recall that a Banach lattice $E$ is called an $AM$-space if for all positive $x,y\in E$, $\|x\vee y\|= \|x\|\vee \|y\|$.
\begin{prop}\label{800}
Suppose $E$ is a Banach lattice. Then, $E$ is an $AM$-space if and only if it possesses $AM$-property.
\end{prop}
\begin{proof}
The direct implication is trivial. For the other direction, assume that $E$ has $AM$-property. Consider bounded set $B\subseteq E$. W.L.O.G, suppose that $B=\{x\in E, \|x\|<1\}$. By the assumption, $B^{\vee}$ is also bounded so that there exists a positive real $M>1$ such that for each $x,y \in B_{+}$, $\|x\vee y\|<M$. Now, we use this simple inequality in reals; given reals $a,b$. If for every real $c$, $a<c$ implies that $b<c$, then $b\leq a$. Thus, $\|x\vee y\|\leq \|x\|\vee \|y\|$. Moreover, it is easy to see that $\|x\|\vee\|y\|\leq \|x\vee y\|$.
\end{proof}
By Proposition \ref{800}, the notions of an $AM$-space and the $AM$-property in a Banach lattice agree. But this is not the only case in the category of all locally solid vector lattices.

\begin{prop}\label{0}
Suppose $(X_{\alpha})_{\alpha\in A}$ is a family of locally solid vector lattices. Put $X=\prod_{\alpha\in A}{X_{\alpha}}$ with product topology and pointwise ordering. If each $X_{\alpha}$ has the $AM$-property, then so is $X$.
\end{prop}
\begin{proof}
Suppose $B\subseteq X$ is bounded. By a simple modification of \cite[Theorem 3.1]{Z}, there exists a net $(B_{\alpha})_{\alpha\in A}$ such that for each $\alpha$, $B_{\alpha}\subseteq X_{\alpha}$ is bounded and $B\subseteq \prod_{\alpha\in A}B_{\alpha}$. We show that $B^{\vee}$ is also bounded. Let $W$ be an arbitrary zero neighborhood in $X$. So, there are zero neighborhoods $(U_{\alpha_i})_{i\in\{1,\ldots,n\}}$ such that $W=\prod_{i=1}^{n}U_{\alpha_i}\times \prod_{\alpha\in A-\{\alpha_1,\ldots,\alpha_n\}}X_{\alpha}$.

Observe that for each $x\in B$, there is a net $(x_{\beta})_{\beta \in A}$ with $x_{\beta}\in B_{\beta}$. Now, consider the set $\{x_1,\ldots,x_m\}\subseteq B$ in which $m\in \Bbb N$ is fixed but arbitrary. It is enough to show that $x_1\vee\ldots\vee x_m$ is also bounded. Note that
\[x_1\vee\ldots\vee x_m=(x_{\beta}^{1})\vee\ldots\vee(x_{\beta}^{m})=(x_{\beta}^{1}\vee\ldots\vee x_{\beta}^{m})_{\beta\in A}.\]
Where $x_{\beta}^{j}\in B_{\beta}$ for each $j\in\{1,\ldots m\}$. For each $i\in\{1,\ldots n\}$, $B_{\alpha_i}$ has $AM$-property so that there exists an $\alpha_i\in {\Bbb R}_{+}$ with $(x_{\alpha_i}^{1}\vee\ldots\vee x_{\alpha_i}^{m})\in \alpha_iU_{\alpha_i}$. Put $\alpha=\max\{\alpha_1,\ldots,\alpha_n\}$. Then, it can be seen easily that $x_1\vee\ldots\vee x_m\in \alpha W$, as claimed.
\end{proof}
The following result may be known; we present a proof for the sake of completeness.
\begin{prop}\label{15}
Suppose $(X_{\alpha})_{\alpha\in A}$ is a family of locally solid vector lattices. Put $X=\prod_{\alpha\in A}{X_{\alpha}}$ with product topology and pointwise ordering. If each $X_{\alpha}$ has the Levi property, then so is $X$.
\end{prop}
\begin{proof}
Suppose $(x^{\beta})_{\beta \in B}$ is a bounded increasing net in $X$. We need to show that its supremum exists. Observe that for each $\beta$, $x^{\beta}=(x^{\beta}_{\alpha})_{\alpha\in A}$. Since $X$ has product topology, we conclude that the net is pointwise bounded; more precisely, for each fixed $\alpha$, the net $(x^{\beta}_{\alpha})_{\beta\in B}$ is bounded and also increasing in $X_{\alpha}$ so that it has a supremum by the assumption, namely, $y_{\alpha}=\sup\{(x^{\beta}_{\alpha})_{\beta\in B}\}$. Now, it can be easily seen that $y=(y_{\alpha})_{\alpha\in A}=\sup \{(x^{\beta}_{\alpha})_{\alpha\in A,\beta\in B}\}$.
\end{proof}
Proposition \ref{0} and Proposition \ref{15} describe many examples of locally solid vector lattices with $AM$ and Levi properties. For example, consider ${\Bbb R}^{\Bbb N}$, the space of all real sequences. It is a locally solid vector lattice with product topology and pointwise ordering. On may consider this point that it has $AM$ and Levi properties.

Let us recall some notions regarding bounded operators between topological vector spaces. Let $X$ and $Y$ be  topological vector spaces. A linear operator $T$ from $X$ into $Y$ is said to be $nb$-bounded if there is a zero neighborhood $U\subseteq X$ such that $T(U)$ is bounded in $Y$. $T$ is called $bb$-bounded if for each bounded set $B\subseteq X$, $T(B)$ is bounded. These concepts are not equivalent; more precisely, continuous operators are, in a sense,  in the middle of these notions of bounded operators, but in a normed space, these concepts have the same meaning.
The class of all $nb$-bounded operators from $X$  into $Y$ is denoted by $B_{n}(X,Y)$ and is equipped with the topology of uniform convergence on some zero neighborhood, namely, a net $(S_{\alpha})$ of $nb$-bounded operators converges to zero on some zero neighborhood $U\subseteq X$ if for any zero neighborhood $V\subseteq Y$ there is an $\alpha_0$ such that $S_{\alpha}(U) \subseteq V$ for each $\alpha\geq\alpha_0$. The class of all $bb$-bounded operators from $X$ into $Y$ is denoted by $B_{b}(X,Y)$ and is allocated to the topology of uniform convergence on bounded sets. Recall that a net $(S_{\alpha})$ of  $bb$-bounded operators uniformly converges to zero on a bounded set $B\subseteq X$ if for any zero neighborhood $V \subseteq Y$ there is an $\alpha_0$ with $S_{\alpha}(B) \subseteq V$ for each $\alpha\geq\alpha_0$.

The class of all continuous operators from $X$ into $Y$  is denoted by $B_c(X,Y)$ and is equipped with the topology of equicontinuous convergence, namely, a net $(S_{\alpha})$ of continuous operators converges equicontinuously to zero if for each zero neighborhood $V\subseteq Y$ there is a zero neighborhood $U\subseteq X$ such that for every $\varepsilon>0$ there exists an $\alpha_0$ with $S_{\alpha}(U)\subseteq \varepsilon V$ for each $\alpha\geq\alpha_0$. See \cite{Tr} for a detailed exposition on these classes of operators. In general, we have $B_n(X,Y)\subseteq B_c(X,Y)\subseteq B_b(X,Y)$ and when $X$ is locally bounded, they coincide.
\begin{rem}
In general, when we are dealing with bounded operators between locally solid vector lattices, there is no specific relation between these classes of bounded operators and order bounded operators; see \cite{EGZ} for more information. So, it is reasonable to consider $B^{b}_{n}(X,Y)$: the space of all order bounded $nb$-bounded operators, $B^{b}_{b}(X,Y)$: the space of all $bb$-bounded order bounded operators, $B^{b}_{c}(X,Y)$: the space of all continuous order bounded operators between locally solid vector lattices $X$ and $Y$. It is shown in \cite[Lemma 2.2]{EGZ} that these classes of operators under some mild assumptions: order completeness and the Fatou property of the range space, form vector lattices, again. Moreover, with respect to the assumed topology, each class of bounded order bounded operators, is locally solid.
\end{rem}
\begin{thm}\label{12}
  Suppose $X$ is an order complete locally solid vector lattice. The following are equivalent.
  \begin{itemize}
		\item[\em (i)] { $X$ possesses $AM$ and Levi properties}.
		\item[\em (ii)] { Every order bounded set in $X$ is bounded and vice versa}.
		\end{itemize}
\end{thm}
\begin{proof}
$(i)\to (ii)$. The direct implication is trivial by \cite[Theorem 2.19]{AB1} since $X$ is locally solid. For the converse, assume that $B\subseteq X$ is bounded; W.L.O.G, we may assume that $B$ is solid, otherwise, consider the solid hull of $B$ which is again bounded. So, $B_{+}=\{x\in B, x\geq 0\}$ is also bounded. Assume that $(B_{+})^{\vee}$ is the set of all finite suprema of elements of $B_{+}$. By the $AM$-property, $(B_{+})^{\vee}$ is also bounded. In addition, $(B_{+})^{\vee}$ can be considered as an increasing net in $X_{+}$. So, by the Levi property, $\sup (B_{+})^{\vee}$ exists. But in this case, $\sup B_{+}$ also exists and $\sup (B_{+})^{\vee}=\sup B_{+}$. Put $y=\sup B_{+}$. Therefore, for each $x\in B_{+}$, $x\leq y$; now, it is clear from the relation $B\subseteq B_{+}-B_{+}$ that $B$ is also order bounded.

$(ii)\to (i)$. Suppose $B\subseteq X$ is bounded so that order bounded. Now, it is clear that $B^{\vee}$ is also order bounded and therefore bounded.

Suppose $D$ is an upward directed bounded set in $X_{+}$. So, it is order bounded. Now, $D$ has a supremum since $X$ is order complete.

\end{proof}
Observe that order completeness is essential as an assumption for Theorem \ref{12} and can not be removed. Consider $X=C[0,1]$; it possesses $AM$-property. Also, boundedness and order boundedness agree in $X$. But it does not have the Levi property.
\begin{coro}\label{13}
 Suppose $X$ and $Y$ are locally solid vector lattices such that $Y$ possesses $AM$ and  Levi properties. Then every $bb$-bounded operator $T:X\to Y$ is order bounded; similar results hold for $nb$-bounded operators as well as continuous operators.
\end{coro}
\begin{proof}
Suppose $A\subseteq X$ is order bounded. So, it is bounded. By the assumption, $T(A)$ is also bounded in $Y$. Therefore, Theorem \ref{12} results in order boundedness of $T(A)$. The other part follows from this fact that every $nb$-bounded operator as well as every continuous operator is $bb$-bounded.
\end{proof}
By considering Corollary \ref{13} and \cite[Lemma 2.2]{EGZ}, we have the following observations.
\begin{coro}
Suppose $X$ and $Y$ are locally solid vector lattices such that $Y$ possesses $AM$, Fatou, and Levi properties. Then $B_{n}(X,Y)$ is a vector lattice.
\end{coro}

\begin{coro}
Suppose $X$ and $Y$ are locally solid vector lattices such that $Y$ possesses $AM$, Fatou, and Levi properties. Then $B_{b}(X,Y)$ is a vector lattice.
\end{coro}
\begin{coro}
Suppose $X$ and $Y$ are locally solid vector lattices such that $Y$ possesses $AM$, Fatou, and Levi properties. Then $B_{c}(X,Y)$ is a vector lattice.
\end{coro}

\begin{prop}\label{14}
Suppose $X$ is a locally solid vector lattice which possesses $AM$ and Levi properties and $Y$ is any locally solid vector lattice. Then, every order bounded operator $T:X\to Y$  is $bb$-bounded.
\end{prop}

\begin{proof}
Suppose $B\subseteq X$ is bounded. By Theorem \ref{12}, $B$ is also order bounded. By the assumption, $T(B)$ is order bounded in $Y$ so that bounded by considering this point that $Y$ is locally solid.
\end{proof}
\begin{coro}
Suppose $X$ is a locally solid vector lattice which possesses $AM$ and Levi properties and $Y$ is an order complete locally solid vector lattice with the Fatou property. Then, $B^{b}_{b}(X,Y)=B^{b}(X,Y)$.
\end{coro}
\begin{proof}
By Proposition \ref{14}, $B^{b}_{b}(X,Y)\subseteq B^{b}(X,Y)\subseteq B_{b}(X,Y)$. So, $B^{b}_{b}(X,Y)=B^{b}(X,Y)$.

\end{proof}
\begin{rem}
  We can not expect Proposition \ref{14} for either $nb$-bounded operators or continuous operators. Consider the identity operator on ${\Bbb R}^{\Bbb N}$. It is order bounded but not an $nb$-bounded operator; observe that ${\Bbb R}^{\Bbb N}$ has $AM$ and Levi properties by Proposition \ref{0} and Proposition \ref{15}. Furthermore, suppose $X$ is $\ell_{\infty}$ with absolute weak topology and pointwise ordering and $Y$ is $\ell_{\infty}$ with norm topology and pointwise ordering. Then, the identity operator $I$ from $X$ into $Y$ is order bounded but not continuous.
    Again, observe that $X$ has $AM$ and Levi properties: suppose $B\subseteq X$ is absolutely weakly bounded so that norm bounded; since $(\ell_{\infty},\|.\|)$ is a $C(K)$-space, we see that it has $AM$-property. Also, suppose $(x_{\alpha})$ is an increasing absolutely weakly bounded net in $X_{+}$. So, it is norm bounded; we know that $(\ell_{\infty},\|.\|)$ has the Levi property.

    Nevertheless, we have the following observation. Recall that a metrizable locally solid vector lattice $X$ is called a Fr$\acute{e}$chet space if it is complete.
\end{rem}
\begin{prop}
Suppose $X$ is a Fr$\acute{e}$chet space and $Y$ is an order complete locally solid vector lattice with the Fatou property. Then, $B^{b}_{c}(X,Y)=B^{b}(X,Y)$.
\end{prop}
\begin{proof}
By \cite[Theorem 5.19]{AB1}, we conclude that $B^{b}_{c}(X,Y)\subseteq B^{b}(X,Y)\subseteq B_{c}(X,Y)$. Therefore, $B^{b}_{c}(X,Y)=B^{b}(X,Y)$.
\end{proof}

Before, we proceed with an application of $AM$-property, we have the following useful observation. Recall that $B^{b}(X,Y)$ is the space of all order bounded operators from a vector lattice $X$ into a vector lattice $Y$.
\begin{lem}\label{20}
Suppose $X$ and $Y$ are locally solid vector lattices such that $Y$ possesses the Fatou property and is order complete. Then we have the following.
\begin{itemize}
\item[\em (i)] {$B^{b}_{n}(X,Y)$ is an ideal of $B^{b}(X,Y)$}.
\item[\em (ii)]{$B^{b}_{b}(X,Y)$ is an ideal of $B^{b}(X,Y)$}.
\item[\em (iii)] {$B^{b}_{c}(X,Y)$ is an ideal of $B^{b}(X,Y)$}.
\end{itemize}
\end{lem}
\begin{proof}
$(i)$. Assume $|T|\leq |S|$ where $T$ is order bounded and $S\in B^{b}_{n}(X,Y)$. There exists a zero neighborhood $U\subseteq X$ such that $S(U)$ is bounded. So, for each zero neighborhood $V\subseteq Y$, there is a positive scalar $\gamma$ with $S(U)\subseteq \gamma V$. Since $U$ is solid, for any $y\in U$, $y^{+}, y^{-}, |y|\in U$. Fix any $x\in U_{+}$. Then $|T|(x)\leq |S|(x)$. In addition, by the Riesz-Kantorovich formulae, $|S|(x)=\sup\{|S(u)|:|u|\leq x\}$. Since $U$ is solid and $V$ is order closed, we conclude that $|S|(x)\in \gamma V$ so that $|T|(x)\in \gamma V$. Since $|T(x)|\leq |T|(x)$, we see that $|T(x)|\in\gamma V$. So, $T(x)\in\gamma V$. Therefore, $T(U_{+})\subseteq \gamma V$. Since $U\subseteq U_{+}-U_{+}$, we conclude that $T(U)$ is also bounded.

$(ii)$. It is similar to the proof of $(i)$. Just, observe that for a bounded set $B\subseteq X$, W.L.O.G, we may assume that $B$ is solid; otherwise, consider the solid hull of $B$ which is also bounded.

$(iii)$. Assume $|T|\leq |S|$ where $T$ is order bounded and $S\in B^{b}_{c}(X,Y)$. Choose arbitrary zero neighborhood $W\subseteq Y$. There is a zero neighborhood $V$ with $V-V\subseteq W$. Find any neighborhood $U$ such that $S(U)\subseteq V$. Fix any $x\in U_{+}$. Then, $|T|(x)\leq |S|(x)$. In addition, by the Riesz-Kantorovich formulae, $|S|(x)=\sup\{|S(u)|:|u|\leq x\}$. Since $U$ is solid and also $V$ and $W$ are order closed, we conclude that $|S|(x)\in V$ so that $|T|(x)\in V$. Since $|T(x)|\leq |T|(x)$, we see that $|T(x)|\in V$. So, $T(x)\in V$. Therefore, $T(U_{+})\subseteq  V$. Since $U\subseteq U_{+}-U_{+}$, we conclude that $T(U)\subseteq T(U_{+})-T(U_{+})\subseteq V-V\subseteq W$, as desired.
\end{proof}
As a consequence, we state a domination property for each class of bounded order bounded operators.
\begin{coro}
Suppose $X$ and $Y$ are locally solid vector lattices such that $Y$ possesses the Fatou property and is order complete. Moreover, assume that $T,S:X\to Y$ are operators such that $0\leq T\leq S$. Then we have the following.
\begin{itemize}
\item[\em (i)] {If $S\in B^{b}_{n}(X,Y)$ then  $T\in B^{b}_{n}(X,Y)$}.
\item[\em (ii)]{If $S\in B^{b}_{b}(X,Y)$ then  $T\in B^{b}_{b}(X,Y)$}.
\item[\em (iii)] {If $S\in B^{b}_{c}(X,Y)$ then  $T\in B^{b}_{c}(X,Y)$}.
\end{itemize}
\end{coro}
\begin{rem}
We have seen in Lemma \ref{20} that each class of bounded order bounded operators is an ideal in the space of all order bounded operators. So, one interesting question is to ask whether or not the assumed space of operators forms a band. The answer is negative:

Suppose $X$ is ${\Bbb R}^{\Bbb N}$ with product topology and pointwise ordering and $P_n$ is the $n$-th projection on $X$. Then, each $P_n$ is $nb$-bounded as well as order bounded. In addition, $P_n\uparrow I$, where $I$ is the identity operator on $X$. But $I$ is not $nb$-bounded.
Furthermore, suppose $X$ is $\ell_{\infty}$ with absolute weak topology and pointwise ordering and $Y$ is $\ell_{\infty}$ with norm topology and pointwise ordering. Again, assume that $P_n$ is the $n$-th projection from $X$ into $Y$. It is easy to see that each $P_n$ is continuous as well as order bounded; moreover, $P_n\uparrow I$, where $I$ is the identity operator from $X$ into $Y$. But $I$ is not continuous, certainly.

\end{rem}

\begin{thm}
Suppose $X$ is a locally solid-convex vector lattice and $Y$ is an order complete locally solid vector lattice with the Fatou property. Then $B^{b}_{n}(X,Y)$ has the Levi property if and only if so is $Y$.
\end{thm}
\begin{proof}
Suppose $(T_{\alpha})$ is a bounded increasing net in ${B^{b}_{n}(X,Y)}_{+}$. This implies that there is a zero neighborhood $U\subseteq X$ such that $(T_{\alpha}(U))$ is bounded for each $\alpha$. So, for each $x\in X_{+}$, the net $(T_{\alpha}(x))$ is bounded and increasing in $Y_{+}$ so that it has a supremum, namely, $\alpha_x$. Define $T_{\alpha}:X_{+}\to Y_{+}$ via $T_{\alpha}(x)=\alpha_x$. It is an additive map; it is easy to see that $\alpha_{x+y}\leq \alpha_x+\alpha_y$. For the converse, fix any $\alpha_0$. For each $\alpha\geq\alpha_0$, we have $T_{\alpha}(x)\leq \alpha_{x+y}-T_{\alpha}(y)\leq \alpha_{x+y}-T_{\alpha_0}(y)$ so that $\alpha_x\leq \alpha_{x+y}-T_{\alpha_0}(y)$. Since $\alpha_0$ was arbitrary, we conclude that $\alpha_{x}+\alpha_{y}\leq\alpha_{x+y}$. By \cite[Theorem 1.10]{AB}, it extends to a positive operator $T:X\to Y$. It is enough to show that $T\in B^{b}_{n}(X,Y)$. It is clear that $T$ is order bounded. Suppose $V$ is an arbitrary zero neighborhood in $Y$. There is a positive scalar $\gamma$ with $T_{\alpha}(U)\subseteq \gamma V$. This means that $T(U)\subseteq \gamma V$ since $V$ is order closed.

For the converse, assume that $(y_{\alpha})$ is a bounded increasing net in $Y_{+}$. Pick any $0\neq x_0\in X_{+}$.  By the Hahn-Banach theorem, there exists $f\in X_{+}^{*}$ such that $f(x_0)=1$. Define $T_{\alpha}:X\to Y$ with $T_{\alpha}(x)=f(x)y_{\alpha}$. It is easy to see that each $T_{\alpha}$ is $nb$-bounded as well as order bounded. There exists a zero neighborhood $U\subseteq X$ such that $|f(x)|\leq 1$ for each $x\in U$. It follows that $(T_{\alpha})$ is bounded and increasing. Thus, by the assumption, $T_{\alpha}\uparrow T$ for some $T\in B^{b}_{n}(X,Y)$. Therefore, $T_{\alpha}(x_0)\uparrow T(x_0)$; that is $y_{\alpha}\uparrow T(x_0)$, as claimed.
\end{proof}
\begin{thm}
Suppose $X$ is a locally solid-convex vector lattice and $Y$ is an order complete locally solid vector lattice with the Fatou property. Then $B^{b}_{b}(X,Y)$ has the Levi property if and only if so is $Y$.
\end{thm}
\begin{proof}
Suppose $(T_{\alpha})$ is a bounded increasing net in ${B^{b}_{b}(X,Y)}_{+}$. Fix a bounded set $B\subseteq X$. This implies that the set $(T_{\alpha}(B))$ is bounded for each $\alpha$. So, for each $x\in X_{+}$, the net $(T_{\alpha}(x))$ is bounded and increasing in $Y_{+}$ so that it has a supremum, namely, $\alpha_x$. Define $T_{\alpha}:X_{+}\to Y_{+}$ via $T_{\alpha}(x)=\alpha_x$. It is an additive map. By \cite[Theorem 1.10]{AB}, it extends to a positive operator $T:X\to Y$. It is sufficient to show that $T\in B^{b}_{b}(X,Y)$. It is clear that $T$ is order bounded. Suppose $V$ is an arbitrary zero neighborhood in $Y$. There is a positive scalar $\gamma$ with $T_{\alpha}(B)\subseteq \gamma V$. This means that $T(B)\subseteq \gamma V$ since $V$ is order closed.

For the converse, assume that $(y_{\alpha})$ is a bounded increasing net in $Y_{+}$. Pick any $0\neq x_0\in X_{+}$.  By the Hahn-Banach theorem, there exists $f\in X_{+}^{*}$ such that $f(x_0)=1$. Define $T_{\alpha}:X\to Y$ with $T_{\alpha}(x)=f(x)y_{\alpha}$. It is easy to see that each $T_{\alpha}$ is $bb$-bounded as well as order bounded. Fix a bounded set $B\subseteq X$. Without loss of generality, we may assume that $|f(x)|\leq 1$ for each $x\in B$. It follows that $(T_{\alpha})$ is bounded and increasing. Thus, by the assumption, $T_{\alpha}\uparrow T$ for some $T\in B^{b}_{b}(X,Y)$. Therefore, $T_{\alpha}(x_0)\uparrow T(x_0)$; that is $y_{\alpha}\uparrow T(x_0)$, as desired.
\end{proof}
\begin{thm}
Suppose $X$ is a locally solid-convex vector lattice and $Y$ is an order complete locally solid vector lattice with the Fatou property. Then $B^{b}_{c}(X,Y)$ has the Levi property if and only if so is $Y$.
\end{thm}
\begin{proof}
Suppose $(T_{\alpha})$ is a bounded increasing net in ${B^{b}_{c}(X,Y)}_{+}$. This implies that the net $(T_{\alpha})$ is equicontinuous. So, for any zero neighborhood $V\subseteq Y$, there is a zero neighborhood $U\subseteq X$ such that $T_{\alpha}(U)\subseteq V$ for each $\alpha$. By \cite[Theorem 2.4]{Ru}, for each $x\in X_{+}$, the net $(T_{\alpha}(x))$ is bounded and increasing in $Y_{+}$ so that it has a supremum, namely, $\alpha_x$. Define $T_{\alpha}:X_{+}\to Y_{+}$ via $T_{\alpha}(x)=\alpha_x$. It is an additive map. By \cite[Theorem 1.10]{AB}, it extends to a positive operator $T:X\to Y$. It suffices to prove that $T\in B^{b}_{c}(X,Y)$. It is clear that $T$ is order bounded. One may verify that $T(U)\subseteq V$ since $V$ is order closed.

For the converse, assume that $(y_{\alpha})$ is a bounded increasing net in $Y_{+}$. Pick any $0\neq x_0\in X_{+}$.  By the Hahn-Banach theorem, there exists $f\in X_{+}^{*}$ such that $f(x_0)=1$. Define $T_{\alpha}:X\to Y$ with $T_{\alpha}(x)=f(x)y_{\alpha}$. It is easy to see that each $T_{\alpha}$ is continuous as well as order bounded. There exists a zero neighborhood $U\subseteq X$ such that $|f(x)|\leq 1$ for each $x\in U$. It follows that $(T_{\alpha})$ is equicontinuous (bounded) and increasing. Thus, by the assumption, $T_{\alpha}\uparrow T$ for some $T\in B^{b}_{c}(X,Y)$. Therefore, $T_{\alpha}(x_0)\uparrow T(x_0)$; that is $y_{\alpha}\uparrow T(x_0)$, as claimed.
\end{proof}

\begin{prop}
Suppose $X$ and $Y$ are locally solid vector lattices such that $X$ is locally convex and $Y$ has the Fatou property and is order complete. If $B^{b}_{n}(X,Y)$ has the Lebesgue property, then so is $Y$.
\end{prop}
\begin{proof}
Suppose $(y_{\alpha})$ is a net in $Y$ such that $y_{\alpha}\downarrow 0$. Pick any positive $x_0\in X$. By the Hahn-Banach theorem, there exists $f\in X_{+}^{*}$ such that $f(x_0)=1$. Define $T_{\alpha}:X\to Y$ with $T_{\alpha}(x)=f(x)y_{\alpha}$. It is easy to see that each $T_{\alpha}$ is $nb$-bounded as well as order bounded.
Note that by Lemma \ref{20}, $B^{b}_{n}(X,Y)$ is an ideal in $B^{b}(X,Y)$. Thus by considering \cite[Theorem 1.35]{AB} and \cite[Theorem 1.18]{AB}, we conclude that $T_{\alpha}\downarrow 0$ in $B^{b}_{n}(X,Y)$.
So, by the assumption, $T_{\alpha}\rightarrow 0$ uniformly on some zero neighborhood $U\subseteq X$. Therefore, $T_{\alpha}(x_0)\rightarrow 0$ in $Y$; this means $(y_{\alpha})$ is a null net in $Y$, as we wanted.

\end{proof}

\begin{prop}
Suppose $X$ and $Y$ are locally solid vector lattices such that $X$ is locally convex and $Y$ possesses the Fatou property and is order complete. If $B^{b}_{b}(X,Y)$ has the Lebesgue property, then so is $Y$.
\end{prop}
\begin{proof}
Suppose $(y_{\alpha})$ is a net in $Y$ such that $y_{\alpha}\downarrow 0$. Pick any positive $x_0\in X$. By the Hahn-Banach theorem, there exists $f\in X_{+}^{*}$ such that $f(x_0)=1$. Define $T_{\alpha}:X\to Y$ via $T_{\alpha}(x)=f(x)y_{\alpha}$. It is easy to see that each $T_{\alpha}$ is $bb$-bounded and order bounded.
Note that by Lemma \ref{20}, $B^{b}_{b}(X,Y)$ is an ideal in $B^{b}(X,Y)$. Therefore, \cite[Theorem 1.35]{AB} and \cite[Theorem 1.18]{AB} yield that $T_{\alpha}\downarrow 0$ in $B^{b}_{b}(X,Y)$.
So, by the assumption, $T_{\alpha}\rightarrow 0$ uniformly on bounded sets. Therefore, $T_{\alpha}(x_0)\rightarrow 0$ in $Y$; this means $(y_{\alpha})$ is a null net in $Y$, as claimed.

\end{proof}
For the converse, we have the following.
\begin{thm}\label{1}
Suppose $X$ and $Y$ are locally solid vector lattices such that $X$ possesses $AM$ and Levi properties and $Y$ is order complete. If $Y$ has the Lebesgue property, then so is $B^{b}(X,Y)$.
\end{thm}
\begin{proof}
First, observe that by Proposition \ref{14}, $B^{b}(X,Y)=B^{b}_{b}(X,Y)$. Suppose $(T_{\alpha})_{\alpha\in I}$ is a net in $B^{b}_{b}(X,Y)$ such that $T_{\alpha}\downarrow 0$.
Choose a bounded set $B\subseteq X$; W.L.O.G, we may assume that $B$ is solid, otherwise, consider the solid hull of $B$ which is certainly bounded. By Theorem \ref{12}, $B$ is order bounded. Put $A=\{T_{\alpha}(x), \alpha\in I, x\in B_{+}\}$. Again, W.L.O.G, assume that $B_{+}=[0,u]$, in which $u\in X_{+}$. Define $\Lambda=I\times [0,u]$. Certainly, $\Lambda$ is a directed set while we consider it with the lexicographic order, namely, $(\alpha,x)\leq (\beta,y)$ if $\alpha<\beta$ or $\alpha=\beta$ and $x\leq y$. In notation, $A=(y_{\lambda})_{\lambda\in \Lambda}\geq {\sf 0}$. So, by considering $A^{\wedge}$, one can assume $A$ as a decreasing  net in $Y_{+}$. Therefore, it has an infimum. We claim that $A\downarrow 0$; otherwise, there is a $0\neq y\in Y_{+}$ such that $y_{\lambda}\geq y$ for each $\lambda\in \Lambda$. Therefore, for each $\alpha$ and each $x\in B_{+}$, $T_{\alpha}(x)\geq y$ which is in contradiction with $T_{\alpha}\downarrow 0$. By the assumption, $y_{\lambda}\rightarrow 0$ in $Y$. Therefore, for an arbitrary zero neighborhood $V\subseteq Y$, there exists a $\lambda_0=(\alpha_0,x_0)$ such that $y_{\lambda}\in V$ for each $\lambda\geq\lambda_0$. Suppose $\lambda=(\alpha,x)$. So, for each $\alpha>\alpha_0$ and for each $x\in B_{+}$,  $T_{\alpha}(x)\in V$. Since $B\subseteq B_{+}-B_{+}$, we conclude that  $T_{\alpha}\rightarrow 0$ in $B^{b}_{b}(X,Y)$.
\end{proof}
\begin{rem}
Observe that hypotheses in Theorem \ref{1} are essential and can not be removed. Consider $X=c_0$ with norm topology; it possesses the $AM$-property and its topology is Lebesgue but it fails to have the Levi property. Suppose $(P_n)$ is the sequence of basis projections on $X$. Each $P_n$ is $bb$-bounded and $P_n\uparrow I$, where $I$ is the identity operator on $X$. But $P_n\nrightarrow I$ uniformly on the unit ball of $X$.
Moreover, consider $Y=\ell_1$ with norm topology; it has the Lebesgue and the Levi properties but it fails to have the $AM$-property. Again, if $(P_n)$ is the sequence of basis projections on $Y$, $P_n\uparrow I$ but certainly not in the topology  of uniform convergence on bounded sets.
\end{rem}
\begin{prop}
Suppose $X$ and $Y$ are locally solid vector lattices such that $X$ is locally convex and  $Y$ has the Fatou property and is order complete. If $B^{b}_{c}(X,Y)$ has the Lebesgue property, then so is $Y$.
\end{prop}
\begin{proof}
Suppose $(y_{\alpha})$ is a net in $Y$ such that $y_{\alpha}\downarrow 0$. Pick any positive $x_0\in X$. By the Hahn-Banach theorem, there exists $f\in X_{+}^{*}$ such that $f(x_0)=1$. Define $T_{\alpha}:X\to Y$ with $T_{\alpha}(x)=f(x)y_{\alpha}$. It is easy to see that each $T_{\alpha}$ is continuous as well as order bounded.
Note that by Lemma \ref{20}, $B^{b}_{c}(X,Y)$ is an ideal in $B^{b}(X,Y)$. Therefore, \cite[Theorem 1.35]{AB} and \cite[Theorem 1.18]{AB} imply that $T_{\alpha}\downarrow 0$ in $B^{b}_{c}(X,Y)$.
So, by the assumption, $T_{\alpha}\rightarrow 0$ equicontinuously; for each zero neighborhood $V\subseteq Y$ there is a zero neighborhood $U\subseteq X$ such that for each $\varepsilon>0$, there exists an $\alpha_0$ with $T_{\alpha}(U)\subseteq \varepsilon V$ for each $\alpha\geq\alpha_0$.  This means that $T_{\alpha}(x)\rightarrow 0$ in $Y$ for each $x$. In particular, $T_{\alpha}(x_0)\rightarrow 0$ in $Y$; this results in $y_{\alpha}\rightarrow 0$ in $Y$.

\end{proof}

In general, we do not know if a version Theorem \ref{1} holds for continuous operators with respect to the equicontinuous convergence topology.
But we can have the following partial observation. Suppose $X$ is a Fr$\acute{e}$chet space which possesses $AM$ and Levi properties, then
by Proposition \ref{14} and \cite[Theorem 5.19]{AB1}, $B^{b}(X,Y)=B^{b}_{c}(X,Y)=B^{b}_{b}(X,Y)$ for any order complete locally solid vector lattice $Y$. So, by using Theorem \ref{1}, we have the following.
\begin{coro}
Suppose $X$ is a Fr$\acute{e}$chet space which possesses $AM$ and Levi properties and $Y$ is an order complete locally solid vector lattice. If $Y$ has the Lebesgue property, then so is $B^{b}_{c}(X,Y)$.
\end{coro}

Moreover, since $B_{c}(X,Y)$ can be viewed as a subspace of $B_{b}(X,Y)$, we can consider the induced topology on it. So, we have the following.
\begin{prop}
Suppose $X$ is a locally solid vector lattice with the Heine-Borel property and $Y$ is a locally solid vector lattice which is order complete. If $Y$ has the Lebesgue property, then so is $B^{b}_{c}(X,Y)$; while it is equipped with the topology of uniform convergence on bounded sets.
\end{prop}
\begin{proof}
Suppose $(T_{\alpha})$ is a net in $B^{b}_{c}(X,Y)$ such that $T_{\alpha}\downarrow 0$. Thus, for any $x\in X_{+}$, $T_{\alpha}(x)\downarrow 0$ so that $T_{\alpha}(x)\rightarrow  0$. Fix a bounded set $B\subseteq X$. Since the net $(T_{\alpha})$ is order bounded, we conclude that it is uniformly bounded on $\overline{B}$. But by the Heine-Borel property, $\overline{B}$ is compact. Now, \cite[Corollary 15]{T} implies that $T_{\alpha}(\overline{B})\rightarrow 0$ in $Y$. This means that $B^{b}_{c}(X,Y)$ possesses the Lebesgue property.
\end{proof}

\end{document}